\documentclass[12pt,a4paper]{amsart}

\usepackage{hyperref}
\usepackage{color}

\usepackage[dvipdfm]{graphicx}
\usepackage{pb-diagram,pb-xy}

\usepackage{amsmath,amsfonts,amssymb}

\usepackage{latexsym,mathrsfs,amsthm, amscd, url, psfrag}
\usepackage[all,graph]{xy}
\usepackage{pb-diagram}
\xyoption{import}

\def\Z{\mathbb{Z}}
\def\N{{\mathbb{N}}}
\def\Q{{\mathbb{Q}}}

\def\wt{\widetilde}

\def\a{\alpha}

\def\g{\gamma}
\def\G{\Gamma}

\def\ac2{\mathcal{AC}_2}
\def\F{\mathbb{F}}

\def\ba{\begin{array}}
\def\ea{\end{array}}
\def\bn{\begin{enumerate}}
\def\en{\end{enumerate}}
\def\fp{F \times P}
\def\ll{\langle}
\def\rr{\rangle}
\theoremstyle{plain}
\newtheorem{theorem}{Theorem}[section]
\newtheorem{proposition}[theorem]{Proposition}
\newtheorem{lemma}[theorem]{Lemma}
\newtheorem{corollary}[theorem]{Corollary}

\newtheorem*{maintheorem}{Main Theorem}

\theoremstyle{definition}
\newtheorem{definition}[theorem]{Definition}
\newtheorem{remark}[theorem]{Remark}

\newtheorem*{fact*}{Fact}
\newtheorem*{claim}{Claim}
\def\sm{\setminus}

\DeclareMathOperator{\im}{im}

\DeclareMathOperator{\Aut}{Aut}

\DeclareMathOperator{\Id}{Id}

\DeclareMathOperator{\GL}{GL}

\newcommand{\eps}{\varepsilon}
\begin{document}

\title[Injectivity for Casson-Gordon type representations]{An Injectivity Theorem for Casson-Gordon Type Representations relating to the Concordance of Knots and Links}


\author{Stefan Friedl}
\address{Mathematisches Institut\\ Universit\"at zu K\"oln\\   Germany}
\email{sfriedl@gmail.com}

\author{Mark Powell}
\address{University of Edinburgh, United Kingdom}
\email{M.A.C.Powell@sms.ed.ac.uk}

\def\subjclassname{\textup{2000} Mathematics Subject Classification}
\expandafter\let\csname subjclassname@1991\endcsname=\subjclassname
\expandafter\let\csname subjclassname@2000\endcsname=\subjclassname


\begin{abstract}
In the study of homology cobordisms, knot concordance and link concordance, the following  technical problem arises frequently:
let $\pi$ be a group and let $M \to N$ be a homomorphism between projective $\Z[\pi]$-modules such that  $\Z_p \otimes_{\Z[\pi]} M\to \Z_p \otimes_{\Z[\pi]} N$ is injective; for which other right $\Z[\pi]$-modules $V$ is the induced map  $V \otimes_{\Z[\pi]} M\to V\otimes_{\Z[\pi]}N$ also injective?
Our main theorem gives a new criterion which combines and generalizes many previous results.
\end{abstract}

\maketitle

\section{Introduction}

Let $\pi$ be a group. Throughout this paper we will assume that all groups are finitely generated.
We denote the augmentation map $\Z[\pi]\to \Z$ by $\eps$ and we denote the composition of $\eps$ with the map $\Z \to \Z_p = \Z/p\Z$ by $\eps_p$.
We can then view $\Z_p$ as a right $\Z[\pi]$-module via the augmentation map.
We are interested in the following question.  Let  $M\to N$ be a homomorphism between projective left $\Z[\pi]$-modules such that the induced map $\Z\otimes_{\Z[\pi]} M\to \Z\otimes_{\Z[\pi]}N$ is injective.  For which other right $\Z[\pi]$-modules $V$ is the induced map  $V\otimes_{\Z[\pi]} M\to V\otimes_{\Z[\pi]}N$ also injective?
Finding such modules is a key ingredient in the obstruction theory for homology cobordisms, knot concordance and link concordance.
We refer to \cite[Proposition~2.10]{COT} and our Section \ref{Section:application} for more information.

To state the previously known results and our main theorem we need to introduce more notation.
Let  $\phi\colon\pi\to H$ be a homomorphism to a torsion-free abelian group, let $Q$ be a field of characteristic zero and let $\a\colon\pi\to \GL(k,Q)$ be a representation.  Note that $\a$ and $\phi$ give rise to a right $\Z[\pi]$-module structure on $Q^k\otimes_{Q} Q[H] = Q[H]^k$ as follows:
\[ \ba{rcl} Q^k\otimes_{Q} Q[H] \times \Z[\pi]&\to & Q^k\otimes_{Q} Q[H]  \\
(v\otimes p,g)&\mapsto & (v\cdot \a(g)\otimes p\cdot \phi(g)).\ea \]
(Here, and throughout this paper we view $Q^k$ as row vectors.)


Given a prime $p$ we say that a group is $p$-group if the order of the group is a power of $p$.
The following is a folklore theorem  (see e.g. \cite[Lemma~3.2]{cha_twisted_2010} for a proof).

\begin{theorem}\label{oldthm}
Let $\pi$ be a group, let $p$ be a prime and let $f \colon M \to N$ be a morphism of projective left $\Z[\pi]$-modules such that
\[\Id\otimes f\colon \Z_p\otimes_{\Z[\pi]} M \to \Z_p\otimes_{\Z[\pi]} N\]
is injective.  Let $\phi \colon \pi \to H$ be an epimorphism onto a torsion-free abelian group and let $\a \colon \pi \to \GL(k,Q)$ be a representation, with $Q$ a field of characteristic zero.  If $\a$ factors through a $p$-group, then
\[\Id\otimes f\colon Q[H]^k\otimes_{\Z[\pi]} M \to Q[H]^k\otimes_{\Z[\pi]} N\]
is also injective.
\end{theorem}

This theorem (sometimes with the special case that $H$ is the trivial group) is at the heart of many results in link concordance.
We refer to the work of Cha, Friedl, Ko, Levine and Smolinski (see  \cite{cha_Hirzebuch_type_defects}, \cite{cha_twisted_2010},  \cite{Fr05},  \cite{CK99}, \cite{CK06}, \cite{Le94}, \cite{Sm89}) for
many instances where the above theorem is implicitly or explicitly used.

Note though that the condition that the representation has to factor through a $p$-group is quite restrictive.
For example it is a consequence of Stallings' theorem \cite{St65} that for a knot any such representation necessarily factors through the abelianization.
The following is now our main theorem:

\begin{maintheorem}[Theorem \ref{Main_Theorem_2}]
Let $\pi$ be a group, let $p$ be a prime and let $f \colon M \to N$ be a morphism of projective left $\Z[\pi]$-modules such that
\[\Id\otimes f\colon \Z_p\otimes_{\Z[\pi]} M \to \Z_p\otimes_{\Z[\pi]} N\]
is injective.  Let $\phi \colon \pi \to H$ be an epimorphism onto a torsion-free abelian group and let $\a \colon \pi \to \GL(k,Q)$ be a representation, with $Q$ a field of characteristic zero.  If $\a|_{\ker (\phi)}$ factors through a $p$-group, then
\[\Id\otimes f\colon Q[H]^k\otimes_{\Z[\pi]} M \to Q[H]^k\otimes_{\Z[\pi]} N\]
is also injective.
\end{maintheorem}

This theorem makes it possible to work with a considerably larger class of representations than in Theorem \ref{oldthm}.
For example the `Casson-Gordon type' representations which appear in \cite{Let00,Fr04,HKL,FV09} and which originate from the seminal paper by Casson and Gordon \cite{CassonGordon} satisfy the hypothesis of Theorem \ref{Main_Theorem_2}.
Our theorem also connects \cite[Theorem~3.2]{cha_Hirzebuch_type_defects} to the Casson-Gordon setup.  An application of our theorem to define so--called algebraic Casson-Gordon obstructions purely in terms of the chain complex with duality of the metabelian cover of the exterior of a knot, will appear in the PhD thesis of the second author \cite{Powellthesis}.
Note that all of the above apply the theorem to the case that the rank of $H$ is one. In this case our result is closely related to the work of Letsche \cite[Propositions~3.3~and~3.9]{Let00}.

In our opinion the most interesting aspect of our theorem is that the group $H$ may have rank greater than one. This allows the possibility of using representations of link complements of Casson-Gordon type, i.e. representations which do not necessarily factor through nilpotent quotients of the link group.  We will investigate this in a future paper.

The paper is organized as follows:  Section \ref{Section:preliminaries} contains some preliminary lemmas which are used in key places in the proof of the main theorem, Section \ref{Section:main_proof} contains the proof of the main theorem, and then Section \ref{Section:application} uses the main theorem to reprove and generalize Casson and Gordon's Lemma 4 in \cite{CassonGordon}.

\subsection*{Acknowledgment.} This research was conducted while MP was a visitor at the Max Planck Institute for Mathematics in Bonn, which he would like to thank for its hospitality.

\section{Preliminaries}\label{Section:preliminaries}

In this section we collect and prove various preliminary lemmas which we will need in the proof of our main theorem.

\subsection{Strebel's class $D(R)$}

\begin{definition}[Strebel \cite{Strebel}]\label{defn:Strebel}
Suppose $R$ is a commutative ring with unity.  We say that a group $G$ lies in Strebel's class $D(R)$ if whenever we have a morphism $f \colon M \to N$ of projective $R[G]$-modules
(not necessarily finitely generated) such that
\[\Id_R \otimes f \colon R \otimes_{R[G]} M \to R \otimes_{R[G]} N\]
 is injective, then $f$ itself is injective.
(Here we view $R$ as an $R[G]$-module via the augmentation $\epsilon \colon R[G] \to R$.)
\qed\end{definition}

The following lemma was proved by Strebel \cite{Strebel} (see also \cite[Lemma 6.8]{chaorr}).

\begin{lemma}\label{Lemma:ChaOrr6.8}
Suppose $G$ is a group admitting a subnormal series
\[G = G_0 \supset G_1 \supset \cdots \supset G_n = \{e\}\]
whose quotients $G_i/G_{i+1}$ are abelian and have no torsion coprime to $p$.  Then $G$ is in $D(\Z_p)$.
\end{lemma}


We will later need the following consequence of the previous lemma:

\begin{lemma}\label{Lemma:GammaDZp}
Let $\pi$ be a group.  Let $\phi \colon \pi \to H$ be an epimorphism onto a  torsion-free abelian group $H$ and let $\a \colon \pi \to \GL(k,Q)$ be a representation, for some field $Q$, such that $\a |_{\ker \phi}$ factors through a $p$-group $P'$.  Let
\[\G := \frac{\pi}{\ker( \a \times \phi \colon \pi \to  \GL(k,Q) \times H)}.\]
Then $\G \in D(\Z_p)$.
\end{lemma}
\begin{proof}
$\G$ admits a normal series
\[\G = \G_0 = \frac{\pi}{\ker(\a \times \phi)} \supset \G_1 = \frac{\ker (\phi)}{\ker(\a \times \phi )} \subset \G_2 = \{e\}.\]
The quotient $\G_0/\G_1 \cong \im (\phi) \cong H$ and is free abelian.  The quotient $\G_1/\G_2 = \G_1$ is isomorphic to $\im(\a|_{\ker(\phi)})$, which is a quotient of a subgroup of the $p$-group $P'$.  Since every subgroup of a $p$-group is a $p$-group by Lagrange's theorem, every quotient is also a $p$-group.  We therefore have that $\G_1$ is a  $p$-group.  Any $p$-group is nilpotent, and its lower central series gives a normal series of subgroups
\[\G_1 = P_0 \supset P_1 \supset P_2 \supset \cdots \supset P_j = \{e\} \]
such that each quotient $P_i/P_{i+1}$ is an abelian $p$-group (See for example \cite[Section~9]{Aschbacher94}).  Combining the two normal series we obtain
\[\G = \G_0 \supset \G_1 = P_0 \supset P_1 \supset P_2 \supset \cdots \supset P_j = \{e\}.\]
Therefore the conditions of Lemma \ref{Lemma:ChaOrr6.8} are satisfied and so $\G \in D(\Z_p)$.
\end{proof}

\subsection{A group theoretic lemma}

The following lemma shows that, roughly speaking the group $\G$ in Lemma \ref{Lemma:GammaDZp} is `virtually a product $F\times P$', where  $F$ is a free abelian group and $P$ is a $p$-group.

\begin{lemma}\label{Lemma:group_theoretic_lemma}
Let $\G$ be a group which admits an epimorphism $\phi\colon\G\to H$ to a free abelian group, such that the kernel is a $p$-group $P$.
Then there exists a finite index subgroup $F$ of $H$ and a monomorphism
\[ \Psi\colon F\times P \to \G\]
with the following properties:
\bn
\item  the injection is the identity on $P$,
\item  the composition
\[ F\to F\times P\to \G\xrightarrow{\phi} H\]
maps $F$ isomorphically onto $F\subset H$,
\item the image of $\Psi$ is a  finite index normal subgroup of  $\G$.
\en
\end{lemma}

\begin{remark}\label{remark:group_theoretic_remark}
In the simplest case that $H = \Z$, such as when we make applications to knot concordance, $\Z$ is free and so $\G$ splits as
\[\G = \Z \ltimes P = \langle t \rangle \ltimes P.\]
To find a finite index normal subgroup which is a product, with the properties in Lemma \ref{Lemma:group_theoretic_lemma}, define $l:= |\Aut(P)|$.   Then
\[\langle t^l \rangle \times P \leq \langle t \rangle \ltimes P\]
is a subgroup as required.  Lemma \ref{Lemma:group_theoretic_lemma} is an extension of this idea.
\end{remark}

\begin{proof}[Proof of Lemma \ref{Lemma:group_theoretic_lemma}]
We consider the following homomorphism
\[ \ba{rcl} \varphi\colon\G & \to & \Aut(P) \\
g&\mapsto & (h\mapsto ghg^{-1}). \ea \]
We denote by $\G'$ its kernel. Note that $\G'\subset \G$ is a normal subgroup of finite index.

We denote by $H'$ the image of $\phi$ restricted to $\G'$. Note that the kernel $P'$ of $\phi$ restricted to $\G'$ is the intersection of $P$ with $\G'$.
Summarizing we have
\[\G' := \ker(\varphi),\;\; H' := \im(\phi|_{\G'})\mbox{ and } P' := \ker(\phi|_{\G'}).\]
From the definition of $\varphi$ it follows immediately that $P'$ is precisely the center of $P$. In particular $P'$ is an abelian group.
It now follows that $\G'$ fits into the following short exact sequence
\[ 1\to P'\to \G'\xrightarrow{\phi} H'\to 0;\]
in particular  $\G'$ is metabelian. Note that by \cite[Theorem~1]{Ha59} any (finitely generated) metabelian group is residually finite.
(Recall that a group $\pi$ is called \emph{residually finite} if
for any non-trivial $g\in \pi$ there exists a homomorphism $\g\colon\pi\to G$ to a finite group $G$ such that $\g(g)$ is non-trivial.)
By applying the definition of residually finite to all non-trivial elements of $P'\subset \G'$ and by taking the direct product of the resulting homomorphisms to finite groups we obtain a homomorphism
$\g\colon\G'\to G$ to a finite group $G$ such that $\g(g)$ is non-trivial for any $g\in P'$.
We now write $\G'':=\ker\{\g\colon \G'\to G\}$. Note that by definition $\G''\cap P'$ is trivial, which implies that the restriction of $\phi$ to $\G''$ is injective.
We also write $F:=\phi(\G'')$.
We summarize the various groups in the following diagram
\[ \ba{cccccccccc}
1&\to & P&\to &\G&\xrightarrow{\phi} & H &\to &0\\
&&\cup &&\cup && \cup&& \\
1&\to & P'&\to &\G'&\xrightarrow{\phi} & H' &\to &0\\
&&\cup &&\cup && \cup&& \\
& & 1&\to &\G''&\xrightarrow{\phi} & F &\to &0.\ea \]
We  thus obtain an injective map  $F\xrightarrow{\phi|_{\G''}^{-1}} \G''\subset \G$ which we denote by $\psi$.
We now consider the following map:
\[ \ba{rcl} \Psi\colon F\times P&\to &\G\\
(f,p)&\mapsto &  \psi(f)\cdot p.\ea \]
We claim that this map has all the desired properties.
First note that this map is indeed a group homomorphism since for any $f\in F$ we have
\[ \psi(f)\in \G'=\{ g\in \G\,|\, gpg^{-1}=p\mbox{ for all }p\in P\}, \]
i.e. the images of $P$ and $F$ commute in $\G$.
The map is furthermore evidently the identity on $P$ and  the concatenation
\[ F\to F\times P\to \G\xrightarrow{\phi} H\]
maps $F$ isomorphically onto $F\subset H$.

Also note that $\Psi$ is injective. Indeed, suppose we have a pair $(f,p)$ such that $\Psi(f,p)$ is trivial. It follows that  $0=\phi(\Psi(f,p))=\phi(\psi(f)p)=f$ which implies that $(f,p)=(0,p)$.
But $\Psi$ is clearly injective on $P$.

Finally note that
\[ \Psi(F\times P)=\ker\{\G\xrightarrow{\phi} H\to H/F\}.\]
Indeed, it is clear that the left hand side is contained in the right hand side. On the other hand, if $g\in\ker\{\G\xrightarrow{\phi} H\to H/F\}$, then
$\phi(g)=f$ for some $f\in F$ and hence $g\psi(f^{-1})\in P$, i.e. $g=\Psi(\psi(f^{-1}),p)$.
It thus follows that the image of $\Psi$ is indeed a finite index normal subgroup of $\G$.
\end{proof}

We now state and prove a few more elementary lemmas which will be useful in the proof of our main theorem.

\begin{lemma}\label{Lemma:Gamma_l}
Let $A$ be a group and let $B \leq A$ be a finite index subgroup of index $l$, with right coset representatives $\{ t_0, t_1,\dots,t_{l-1}\} \subset A$
so that the quotient is $B \backslash A = \{Bt_i \, |\, i= 0,1,\dots,l-1\}$.   Then $\Z[A]$ is free as a left $\Z[B]$-module with basis $T$.
\end{lemma}

\begin{proof}
It is straightforward to verify that the following  homomorphism of left $\Z[B]$-modules is an isomorphism:
\[\ba{rcl} \Z[B]^l & \xrightarrow{\simeq} & \Z[A] \\ (b_0,\dots,b_{l-1}) & \mapsto & b_0t_0 + b_1t_1 + \cdots + b_{l-1}t_{l-1}.
\ea\]
\end{proof}

Let $\phi:\G\to H$ be an epimorphism onto a free abelian group and let $\a:\G \to \GL(k,Q)$ be a representation over a field such that $\a|_{\ker \phi}$ factors through a $p$-group and such that $\ker(\a \times \phi)$ is trivial.  We apply Lemma \ref{Lemma:group_theoretic_lemma} and view the resulting group $F\times P$ as a subgroup of $\G$.  We now have the following lemma.

\begin{lemma}\label{Lemma:keystep}
There exists an isomorphism of right $\Z[\G]$-modules
\[ Q^k\otimes_Q Q[H]\to (Q^k\otimes_Q Q[F])\otimes_{\Z[F\times P]}\Z[\G],\]
where the right $\Z[\G]$-module structure on $Q^k\otimes_Q Q[H]$ is given by $\a\otimes \phi$,
the right $\Z[F\times P]$-module structure on $Q^k\otimes_Q Q[F]$ is given by $\a\otimes \phi$ restricted to $F\times P$,
and the right $\Z[\G]$-structure on $(Q^k\otimes_Q Q[F])\otimes_{\Z[F\times P]}\Z[\G]$ is given by right multiplication.
\end{lemma}

%
%

\begin{proof}
Consider the following map:
\[\ba{rcl} \theta \colon Q[F]^k \otimes_{\Z[F \times P]} \Z[\G] & \to & Q[H]^k \\
(p_1,\dots,p_k) \otimes g & \mapsto & (p_1,\dots,p_k) \cdot g.
\ea\]
To see that this is well-defined, suppose that $g \in \Z[\fp] \subset \Z[\G]$.  Then
\begin{eqnarray*} \theta \left((p_1,\dots,p_k) \otimes gh \right) &=& (p_1,\dots,p_k) \cdot gh \\ & = & (p_1,\dots,p_k) \cdot g \cdot h \\ & = & \theta \left((p_1,\dots,p_k) \cdot g \otimes h \right).\end{eqnarray*}
It is clear that $\theta$  preserves the right $\Z[\G]$-module structure.

We pick representatives $t_0,\dots,t_{l-1}$ for $(F\times P) \backslash \G$. Note that by   Lemma \ref{Lemma:Gamma_l}  $Q[H]$ is free as a  $Q[F]$-module, with basis $\{\phi(t_i) \, | \, i= 0,\dots,l-1\}$. We write  $x_i := \phi(t_i)$.
Given $(h_1,\dots,h_k) \in Q[H]^k$, there exist unique $b_{ij} \in Q[F]$ such that
\[(h_1,\dots,h_k) = \left(\sum_{i=0}^{l-1} b_{i1} x_i,\dots, \sum_{i=0}^{l-1} b_{ik} x_i\right).\]
Then define:
\[\ba{rcl} \rho \colon Q[H]^k & \to & Q[F]^k \otimes_{\Z[F \times P]} \Z[\G]\\
(h_1,\dots,h_k) & \mapsto & \sum_{i=0}^{l-1}(b_{i1},\dots,b_{ik}) \a(t_i^{-1}) \otimes t_i.
\ea\]
It is straightforward to verify, that $\theta \circ \rho$ is the identity, in particular $\theta$ is surjective.
Since  $\theta$ is a $Q[F]$-module homomorphism between
free $Q[F]$-modules of rank $k\cdot [H:F]=k\cdot [\G:(F\times P)]$ it now follows that $\theta$ is in fact an isomorphism.

Alternatively one can easily show that $\rho$ defines an inverse to $\theta$.

\end{proof}

\subsection{Maschke's Theorem}

In order to show that after taking the representation $\a \otimes \phi$, we still have an injective map, we will need some elementary representation theory.

\begin{definition}
We say that a right $S$-module $M$ is \emph{semisimple} if every submodule is a direct summand.
\end{definition}

\begin{theorem}[Maschke, \cite{La02} Theorem XVIII.1.2]\label{Thm:Maschke}
Let $G$ be a finite group and $\mathbb{F}$ a field of characteristic zero. Then $\F[G]$, the group algebra of G, is semisimple.
\end{theorem}

\begin{corollary}\label{Cor:submodule_direct_summand}
Let $\F$ be a field of characteristic zero, let $G$ be a finite group and let $V$ be a finitely generated right $\F[G]$-module.  Then $V$ is isomorphic to a direct summand of $\F[G]^n$ for some $n$.
\end{corollary}

\begin{proof}
Since $V$ is a finitely generated right $\F[G]$-module there exists an $n\in \N$ and an epimorphism $\phi:\F[G]^n\to V$.
We denote by $N$ the kernel of this map. Note that $\F[G]$ is semisimple by Maschke's theorem.
By \cite[Section~XVII~\S2]{La02} a module is semisimple if and only if it is the sum of a family of simple modules.
It follows that $\F[G]^n$ is also semisimple. In particular $N$ is a direct summand of $\F[G]^n$, i.e.  $\F[G]^n \cong N \oplus M$ for some $\F[G]$-module $M$. It is straightforward to verify that $\phi$ restricts to an isomorphism $M\to V$.
\end{proof}

\section{Proof of the Main Theorem}\label{Section:main_proof}

We are now in a position to prove our main theorem. For the reader's convenience we recall the statement:

\begin{theorem}\label{Main_Theorem_2}
Let $\pi$ be a group, let $p$ be a prime and let $f \colon M \to N$ be a morphism of projective left $\Z[\pi]$-modules such that
\[\Id \otimes f\colon \Z_p \otimes_{\Z[\pi]} M \to \Z_p \otimes_{\Z[\pi]} N\]
is injective.  Let $\phi \colon \pi \to H$ be an epimorphism onto a torsion-free abelian group and let $\a \colon \pi \to \GL(k,Q)$ be a representation, with $Q$ a field of characteristic zero.  If $\a|_{\ker (\phi)}$ factors through a $p$-group, then
\[\Id\otimes f\colon Q[H]^k\otimes_{\Z[\pi]} M \to Q[H]^k\otimes_{\Z[\pi]} N\]
is also injective.
\end{theorem}

We quickly summarize the key ideas of the proof:
\bn
\item Using Lemmas \ref{Lemma:group_theoretic_lemma}, \ref{Lemma:Gamma_l}
and \ref{Lemma:keystep} we can reduce the proof of  the Theorem  to the case that $\pi=F\times P$ and $F=H$.
\item The group ring $\Z[F\times P]$ can be viewed as $\Z[F][P]$, i.e. the group ring of the finite group $P$ over the commutative ring $\Z[F]$.
\item The quotient field $Q(F)$ over $Q[F]$ is flat over $\Z[F]$, we can thus tensor up and it suffices to consider modules over $Q(F)[P]$.
\item The conclusion then follows from an application of the results of Strebel and of Maschke's theorem.
\en
We now move on to providing the full details of the argument.

\begin{proof}
We define $\G:=\pi/\ker(\a \times \phi)$. Let $V$ be a right $\Z[\pi]$-module such that $\pi\to \Aut(V)$ factors
 through $\pi\to \G$, then we can view $V$ as a right module over $\Z[\G]$ and we obtain the following commutative diagram
\[ \xymatrix{  V\otimes_{\Z[\pi]} M \ar[d]\ar[r]^f & V\otimes_{\Z[\pi]} N \ar[d]\\
 V\otimes_{\Z[\G]} (\Z[\G] \otimes_{\Z[\pi]} M)\ar[r]^f & V \otimes_{\Z[\G]} (\Z[\G] \otimes_{\Z[\pi]} N).}\]
The vertical maps are easily seen to be  isomorphisms. Also note that $\Z[\G] \otimes_{\Z[\pi]} M$ and $\Z[\G] \otimes_{\Z[\pi]} N$
are projective $\Z[\G]$-modules. Applying these  observations to $V=\Z_p$ and $V=Q[H]^k$ now shows that without loss of generality we can assume that $\G=\pi$.

Since $\G \in D(\Z_p)$ by Lemma \ref{Lemma:GammaDZp}, Definition \ref{defn:Strebel} applies, and we have that
\[\Id \otimes f \colon \Z_p \otimes_{\Z} M \to \Z_p \otimes_{\Z} N\] is injective.
Since $M$ and $N$ are $\Z$-torsion free it follows that $f \colon M \to N$ is also injective.

We now apply Lemma \ref{Lemma:group_theoretic_lemma}. By a slight abuse of notation we view the resulting group $F\times P$ as a subgroup of $\G$.  By Lemma \ref{Lemma:Gamma_l} we can now also view $M$ and $N$ as projective left $\Z[\fp]=\Z[F][P]$-modules.
In particular we can view $M$ and $N$ as $\Z[F]$-modules.
We now write $M'=Q(F)[P] \otimes_{\Z[\fp]} M$, $N'=Q(F)[P] \otimes_{\Z[\fp]} N$ and we consider the following commutative diagram:
\[\xymatrix{
M \ar@^{(->}[r]\ar[r]  \ar@^{(->}[d]\ar[d]& N  \ar@^{(->}[d]\ar[d]\vspace{0.3cm}\\
Q\otimes_\Z M \ar[r]   \ar[d]^\cong &Q\otimes_\Z N \ar[d]^\cong \\
Q[F]\otimes_{\Z[F]} M \ar[r] \ar@^{(->}[d]\ar[d]   &Q[F]\otimes_{\Z[F]} N     \ar@^{(->}[d]\ar[d]\\
Q(F)\otimes_{\Z[F]} M \ar[r] \ar[d]^\cong &Q(F)\otimes_{\Z[F]} N \ar[d]^\cong\\
Q(F)[P]\otimes_{\Z[F][P]} M \ar[r] \ar[d]^= &Q(F)[P]\otimes_{\Z[F][P]} N\ar[d]^=\\
M'\ar[r] & N'. }\]
The horizontal maps are induced by $f$ and the vertical maps are the canonical maps. The top horizontal map is injective by the preceding discussion.
 Since $Q$ is a field of characteristic zero, and so as an abelian group is $\Z$-torsion free, it is flat as a $\Z$-module. Also note that $Q(F)$, the quotient field of $Q[F]$, is flat as a $Q[F]$-module.
 It follows that the top  and the third vertical maps are injective. The second  vertical map is clearly an isomorphism and the fourth is the identity map.
 It now follows that the bottom horizontal map is also injective.

Now recall, as in Lemma \ref{Lemma:keystep}, that $V := Q[F]^k$ is a right $\Z[\fp]$-module via the action induced by $\a \otimes \phi$, and a $Q[F]$-bimodule, with the $Q[F]$ action commuting with the $\Z[F \times P]$ action.  Therefore $V':=Q(F) \otimes_{Q[F]} V$ is a right $Q(F)[P]$-module.  By Corollary \ref{Cor:submodule_direct_summand}, applied with $G = P, $  and $\mathbb{F} = Q(F)$, there is an integer $n \in \mathbb{N}$ and a right $Q(F)[P]$-module $W'$ such that
\[(V' \oplus W') \cong Q(F)[P]^n\]
as right $Q(F)[P]$-modules.

\newpage
We now consider the following commutative diagram:
\[ \xymatrix{
(M')^n \ar@^{(->}[r]\ar[r] & (N')^n  \\
Q(F)[P]^n \otimes_{Q(F)[P]} M'  \ar[r]  \ar[u]^\cong & Q(F)[P]^n \otimes_{Q(F)[P]} N' \ar[u]^\cong \\
(V'\oplus W') \otimes_{Q(F)[P]} M' \ar[r]\ar[u]^\cong & (V'\oplus W') \otimes_{Q(F)[P]} N' \ar[u]^\cong \\
V' \otimes_{Q(F)[P]} M'\oplus W' \otimes_{Q(F)[P]} M'  \ar[u]^\cong \ar[r] & V' \otimes_{Q(F)[P]} N'\oplus W' \otimes_{Q(F)[P]} N'   \ar[u]^\cong \\
V' \otimes_{Q(F)[P]} M' \ar[r] \ar@^{(->}[u]\ar[u] & V' \otimes_{Q(F)[P]} N' \ar@^{(->}[u]\ar[u]  \\
V' \otimes_{Q(F)[P]} Q(F)[P]\otimes_{\Z[F\times P]} M \ar[r] \ar[u]^= & V' \otimes_{Q(F)[P]}  Q(F)[P]\otimes_{\Z[F\times P]} N \ar[u]^=  \\
V'\otimes_{\Z[F\times P]} M \ar[r] \ar[u]^\cong & V' \otimes_{\Z[F\times P]} N \ar[u]^\cong  \\
 Q(F) \otimes_{Q[F]} V \otimes_{\Z[\fp]} M \ar[r] \ar[u]^= & Q(F) \otimes_{Q[F]} V \otimes_{\Z[\fp]} N \ar[u]^=  \\
Q[F] \otimes_{Q[F]} V \otimes_{\Z[\fp]} M \ar[r] \ar@^{(->}[u]^{(*)}\ar[u] & Q[F] \otimes_{Q[F]} V \otimes_{\Z[\fp]} N  \ar@^{(->}[u]^{(*)}\ar[u] \\
 V \otimes_{\Z[\fp]} M \ar[r]\ar[u]^{\cong}& V \otimes_{\Z[\fp]} N \ar[u]^{\cong} \\
(Q^k\otimes_Q Q[F]) \otimes_{\Z[\fp]} \Z[\G]\otimes_{\Z[\G]} M \ar[r]\ar[u]^{\cong}& (Q^k\otimes Q[F]) \otimes_{\Z[\fp]} \Z[\G]\otimes_{\Z[\G]} N \ar[u]^{\cong} \\
   (Q^k\otimes_Q Q[H])\otimes_{\Z[\G]} M \ar[r]\ar[u]^{\cong}&  (Q^k \otimes_Q Q[H])\otimes_{\Z[\G]} N \ar[u]^{\cong} \\
   Q[H]^k\otimes_{\Z[\G]} M \ar[r] \ar[u]^\cong& Q[H]^k\otimes_{\Z[\G]} N \ar[u]^\cong }\]
The horizontal maps are induced by $f$ and  the vertical maps are the canonical maps. Except for the bottom map it is clear from the definitions that these vertical maps are monomorphisms or isomorphisms. It follows from Lemma \ref{Lemma:keystep} that the bottom vertical map is also an isomorphism.
Note that the eighth vertical map $(*)$ is injective since $V\otimes_{\Z[F\times P]} M$ and $V\otimes_{\Z[F\times P]} N$ are projective $Q[F]$-modules and $Q(F)$ is flat over $Q[F]$.
Since the top horizontal map is injective it now follows also that the bottom horizontal map is injective.
   This completes the proof of our main theorem.

%
\end{proof}

\section{Application to Concordance Obstructions}\label{Section:application}

In this section we indicate what we see as an intended application of Theorem \ref{Main_Theorem_2}.  In the theory of knot concordance,
 both the knot exterior and the exterior of a slice disk, should one exist, are $\Z$-homology circles.  To define an obstruction theory one uses the idea of homology surgery.
That is, just as was done by Casson and Gordon \cite{CassonGordon}, we take the zero-framed surgery on a knot $M_K$, and take its
$k$-fold cyclic cover $M_k$, for some prime power $k$.  Let $\phi \colon \pi_1(M_K) \to \Z$ be the abelianisation and suppose we have a representation:
\[\a' \colon \pi_1(M_k) \to \GL(1,Q),\]
for a field $Q$ of characteristic zero, which factors through a $p$-group.
The Casson-Gordon obstructions first find a 4-manifold $W$ whose boundary is $sM_k$ for some integer $s$, over which the representation extends, and then use the Witt class of the intersection form of $W$ with coefficients in $Q(\Z)$ to obstruct the possibility of choosing a $W$ which is the $k$-fold cover of a $\Z$-homology circle.  A vital lemma in constructing this obstruction theory is to show that the twisted homology of a chain complex vanishes with $Q(\Z)^k$ coefficients when the $\Z$-homology vanishes, such as is the case for the relative chain complex $C_*(\wt{Y},\wt{S^1})$, when $Y$ is a $\Z$-homology circle.

The following proposition gives a more general result in this direction.

\begin{proposition}\label{prop:lifting_homology_to_covers}
Suppose that $S,Y$ are finite CW-complexes such that there is a map $i \colon S \to Y$ which induces an isomorphism \[i_* \colon H_*(S;\Z_p) \xrightarrow{\simeq} H_*(Y;\Z_p),\]
which, for example, is always the case if $i$ induces a $\Z$-homology equivalence.
 Let $\phi \colon \pi_1(Y) \to H$ be an epimorphism onto a torsion-free abelian group $H$
 and let $\varphi:H\to A$ be an epimorphism onto a finite abelian group.
Let $Y_\varphi$ be the induced  cover of $Y$,
 let $\phi' \colon \pi_1(Y_\varphi) \to H$ be the restriction of $\phi$, and let $\a' \colon \pi_1(Y_\varphi) \to \GL(d,Q)$ be a $d$-dimensional representation to a field $Q$ of characteristic zero, such that $\a'$ restricted to the kernel of $\phi'$ factors through a $p$-group.  Define $S_\varphi$ to be the pull-back cover $S_\varphi := i^*(Y_\varphi)$.  Then \[ i_* \colon H_*(S_\varphi;Q(H)^d) \xrightarrow{\simeq} H_*(Y_\varphi;Q(H)^d)\]
 is an isomorphism.
\end{proposition}

The main applications are the cases $S=S^1$ and $Y=S^3\sm \nu K$ or $Y$ a slice disk exterior, $\phi$ the epimorphism onto $\Z=\ll t\rr$, $\varphi$ the epimorphism onto $\Z_k$
and $\a':\pi_1(Y_\varphi)\to \Z/p^r\to \GL(1,Q)$ a non-trivial one-dimensional representation which factors through a cyclic $p$-group.
A direct calculation shows that  $H_*(S_\varphi;Q(t))=0$,  it then follows from our proposition that $ H_*(Y_\varphi;Q(H))=0$ as well.
Note though that this conclusion can not be reached by considering $S_\varphi$ and $Y_\varphi$ alone, since  the map $H_*(S_\varphi;\Z_p)\to H_*(Y_\varphi;\Z_p)$ is \emph{not} an isomorphism.

In a future application we will consider the case where $Y$ is the exterior of a 2-component link with linking number one and $S$ is one of the two boundary tori.

\begin{proof}
By  the long exact sequence in homology it suffices to show that  $H_*(Y_{\varphi},S_{\varphi};Q(H)^d)=0$.
In order to apply our main theorem we first have to reformulate the problem.

We write $k:=|A|$.
Note that $\Z[\pi_1(Y)]$ acts by right multiplication on the $kd$-dimensional $Q$-vector space
\[ Q^k\otimes_{\Z[\pi_1(Y_\varphi)]}\Z[\pi_1(Y)].\]
We denote the resulting representation $\pi_1(Y)\to \GL(kd,Q)$ by $\a$.
Note that $\a$ restricted to the kernel of $\phi$ also factors through a $p$-group.
Furthermore note that $\a\otimes \phi$ now induces a representation $\pi_1(Y)\to \GL(kd,Q(H))$.

\begin{claim}
We have  $H_*(Y_{\varphi},S_{\varphi};Q(H)^d)=H_*(Y,S;Q(H)^{kd})$.
\end{claim}

Let $\wt{Y}_{\varphi}$ be the universal cover of $Y_\varphi$ and let $\wt{S}_{\varphi} := i^*(\wt{Y}_{\varphi})$ be the pull-back cover of $S_{\varphi}$. By definition,
\[C_*(Y_{\varphi},S_{\varphi};Q(H)^d) = Q(H)^d \otimes_{\Z[\pi_1(Y_{\varphi})]} C_*(\wt{Y}_{\varphi},\wt{S}_{\varphi}),\]
where we use  the representation $\a' \otimes \phi'$ to equip $Q(H)^d$ with a right $\Z[\pi_1(Y_{\varphi})]$-module structure.  Observe that $Y$ and $Y_{\varphi}$ have the same universal cover, and that this implies that the pull-back covers $\wt{S} := i^*(\wt{Y})$ and $\wt{S}_{\varphi} = i^*(\wt{Y}_{\varphi})$ also coincide.  Therefore
\[Q(H)^d \otimes_{\Z[\pi_1(Y_{\varphi})]} C_*(\wt{Y}_{\varphi},\wt{S}_{\varphi}) = Q(H)^d \otimes_{\Z[\pi_1(Y_{\varphi})]} C_*(\wt{Y},\wt{S}),\]
which moreover is isomorphic to
\[Q(H)^d \otimes_{\Z[\pi_1(Y_{\varphi})]} \Z[\pi_1(Y)] \otimes_{\Z[\pi_1(Y)]} C_*(\wt{Y},\wt{S}).\]
It is now straightforward to verify that the map
\[ \ba{rcl}
(Q^d\otimes_Q Q(H))\otimes_{\Z[\pi_1(Y_{\varphi})]} \Z[\pi_1(Y)] &\to&
Q^k\otimes_{\Z[\pi_1(Y_\varphi)]}\Z[\pi_1(Y)] \otimes Q(H)
\\
(v\otimes h)\otimes g &\mapsto &(v\otimes g)\otimes h\phi(g) \ea \]
induces an isomorphism of $\Z[\pi_1(Y)]$-right modules, with $\a \otimes \phi$ used to define the right $\Z[\pi_1(Y)]$-module structure on the latter module. The claim now follows from the observation that
\[ H_*(Q^k\otimes_{\Z[\pi_1(Y_\varphi)]}\Z[\pi_1(Y)] \otimes Q(H) \otimes_{\Z[\pi_1(Y)]}C_*(\wt{Y},\wt{S}))\cong H_*(Y,S;Q(H)^{kd}).\]

We now write
\[\wt{C}_*:=Q(H)^{kd} \otimes_{\Z[\pi_1(Y)]} C_*(\wt{Y},\wt{S}).\]
By the claim it is enough to show that $\wt{C}_*$ is acyclic.
We now follow a standard argument (see e.g. \cite[Proposition~2.10]{COT}).
Note that chain complexes of finite CW complexes comprise finitely generated free modules.  Define \[C_*:= C_*(Y,S;\Z_p).\]  Since $i_* \colon H_*(S;\Z_p) \xrightarrow{\simeq} H_*(Y;\Z_p)$, we see that $H_*(Y,S;\Z_p) \cong 0$, which implies that $C_*(Y,S;\Z_p)$ is contractible, since it is a chain complex of finitely generated free modules.

Let $\g \colon C_i \to C_{i+1}$ be a chain contraction for $C_*$, so that $d\g_i + \g_{i-1}d = \Id_{C_i}$.
Since $\eps_p$ is surjective and all modules are free, there exists a lift of $\g$, \[\wt{\g} \colon C_{i}(\wt{Y},\wt{S}) \to C_{i+1}(\wt{Y},\wt{S})\]
as shown in the following diagram:
\[\xymatrix  @C+1cm {C_i(\wt{Y},\wt{S}) \ar@{-->}[r]^{\wt{\g}} \ar[d]^{\eps_p} & C_{i+1}(\wt{Y},\wt{S}) \ar[d]^{\eps_p} \\ C_i \ar[r]^{\g} & C_{i+1}.
}\]
  Now,
\[ f:= d\wt{\g} + \wt{\g}d \colon C_{i}(\wt{Y},\wt{S}) \to C_{i}(\wt{Y},\wt{S})\]
is a homomorphism of free $\Z[\pi_1(Y)]$-modules whose augmentation $\eps_p(f)$ is the identity and is therefore invertible.  Note that $f$ is also a chain map, since
\[df - fd = d(d\wt{\g} + \wt{\g}d) - (d\wt{\g} + \wt{\g}d)d = d\wt{\g}d - d\wt{\g}d = 0.\]
Our Theorem \ref{Main_Theorem_2} implies that \[(\a \otimes \phi)(f) \colon \wt{C}_i \to \wt{C}_i\]
is injective, and moreover as an injective morphism between vector spaces of the same dimension is in fact an isomorphism.  Thus, $\wt{\g}$ is a chain homotopy between the chain isomorphism $(\a \otimes \phi)(f)$ and the zero map.  Since a chain isomorphism induces an isomorphism of homology groups, and chain homotopic chain maps induce the same maps on homology, the proof is complete.
\end{proof}

The following corollary is  Casson and Gordon's   Lemma 4 in \cite{CassonGordon}. One of our main motivations in writing this paper was to put this lemma into context and to isolate the algebra involved in the proof of the lemma.
(We refer to \cite[Proposition~3.3]{Let00} for a related result.)

\begin{corollary}\label{cor:CGlemma4}
Let $X$ be a connected infinite cyclic cover of a finite complex $Y$.  Let $\wt{X}$ be a regular $p^r$-fold covering of $X$, with $p$ prime.  If $H_*(Y;\Z_p) \cong H_*(S^1;\Z_p)$, then $H_*(\wt{X};\Q)$ is finite-dimensional.
\end{corollary}

\begin{proof}
 Define
\[\phi \colon \pi_1(Y) \to \pi_1(Y)/\pi_1(X) \xrightarrow{\simeq} \ll t\rr.\]
Let $S=S^1$. We can find a subcomplex $S^1 \subset Y$ which represents an element in $\pi_1(Y)$ which gets sent to $1$ under the map $\phi$.
Note that this $S^1$ necessarily represents a generator of $H_1(Y;\Z_p)$.  Let $i \colon S^1 \hookrightarrow Y$ be the inclusion.
Define $\G := \pi_1(Y)/\pi_1(\wt{X})$. The group $\G$ fits into a short exact sequence
\[1 \to P \to \G \to \ll t\rr \to 1,\]
where $P := \pi_1(X)/\pi_1(\wt{X})$ is a $p$-group with $|P| = p^r$.
 As in Lemma \ref{Lemma:group_theoretic_lemma} and Remark \ref{remark:group_theoretic_remark},
  there is a finite index subgroup $\ll t^l\rr \times P \leq \G$, such that
  \[\ll t^l\rr \to \ll t^l\rr \times P \to \G \xrightarrow{\phi} \ll t\rr\] maps $\ll t^l\rr$ isomorphically onto its image.
   Define $\varphi \colon \ll t\rr =\Z \to \Z/l\Z =:A$ to be the quotient map.
     Let $Y_{\varphi}$ be the induced cover of $Y$, then we have an epimorphism \[\pi_1(Y_{\varphi}) \to\ll t^l\rr \times P.\] Then \[\phi' \colon \pi_1(Y_{\varphi}) \to \ll t^l\rr \times P \to \ll t^l\rr \to \ll t\rr\] is the restriction of $\phi$.  Since $\pi_1(Y_{\phi})$ also projects to $P$, define \[\a'\colon \pi_1(Y_{\phi}) \to P \to \GL(p^r,\Q)\] via the regular representation $P \to \Aut_{\Q}(\Q[P])$: we take $d := p^r$ and $Q := \Q$.

We can then apply Proposition \ref{prop:lifting_homology_to_covers}, to see that
\[H_*(S^1_{\varphi};\Q(t)^d) \cong H_*(Y_{\varphi};\Q(t)^d).\]
An elementary calculation shows that $H_*(S^1_{\varphi};\Q(t)^d)=0$.
This implies that:
\[H_*(Y_{\varphi};\Q(t)^d) \cong 0,\]
which then implies, since $\Q(t)$ is flat over $\Q[t^{\pm 1}]$, that $H_*(Y_{\varphi};\Q[t^{\pm 1}]^d)$ is torsion over $\Q[t^{\pm 1}]$, i.e. finite dimensional over $\Q$.
The corollary now follows from the observation that
\[H_*(Y_{\varphi};\Q[t^{\pm 1}]^d) \cong H_*(Y_{\varphi};\Q[P][t^{\pm 1}]) \cong H_*(\wt{X};\Q).\]
\end{proof}

\bibliographystyle{annotate}
\bibliography{markbib1}

\end{document}